\documentclass{amsart}
\usepackage{amsmath,amsfonts,amsthm}
\usepackage{amssymb,latexsym}
\usepackage{graphics}
\usepackage[colorlinks]{hyperref}

\usepackage{amssymb}

\theoremstyle{plain}
\theoremstyle{definition}
\newtheorem{theorem}{Theorem}[section]
\newtheorem{lemma}[theorem]{Lemma}

\newtheorem{example}[theorem]{Example}

\newtheorem{note}[theorem]{Note}

\theoremstyle{remark}

\numberwithin{equation}{section}
\newcommand{\SP}{\: \: \: }

\title{Closed graph property and Khalimsky spaces}
\author[M. Pourattar, F. Ayatollah Zadeh Shirazi, M. R. Mardanbeigi]{Mehrnaz Pourattar, Fatemah Ayatollah Zadeh Shirazi, Mohammad Reza Mardanbeigi}
\begin{document}
\begin{abstract}
In the following text for Khalimsky $n-$dimensional space $\mathcal{K}^n$ we show self--map
$f:\mathcal{K}^n\to\mathcal{K}^n$ has closed graph if and only if there exist integers $\lambda_1,\ldots,\lambda_n$ 
such that $f$ is a constant map with value $(2\lambda_1,\cdots,2\lambda_n)$. We also show each self--map on Khalimsky circle and
Khalimsky sphere which has closed graph is a constant map. The text is motivated by examples. 
\end{abstract}
\maketitle
\noindent {\small {\bf 2020 Mathematics Subject Classification:}  54C35, 54D05  \\
{\bf Keywords:}}  Alexandroff space, Closed graph, Digital topology, Khalimsky spaces.
\section{Introduction}
\noindent Alexandroff spaces introduced as ``Diskrete R\"{a}ume'' by P. Alexandroff in \cite{alexandroff}.
One may consider different sub--classes of Alexandroff spaces like finite topological 
\linebreak
spaces~\cite{finite},
functional Alexandroff topological spaces~\cite{functional}, locally finite topological spaces~\cite{one, locally},
$n-$dimensional Khalimsky spaces~\cite{semi, one}, etc..  Interaction between Alexandroff topological spaces 
and other mathemathical structures has been studied in different texts 
(see e.g.~\cite{mesal}). 
\\
Closed graph property has been studied in different categories of topological spaces:
Banach spaces, complete metric spaces, and general topological spaces, etc. (see e.g.~\cite{close1, close2, close3, close4, close5}. 
\\
Let's recall that in topological spaces $X,Y$ and map $f:X\to Y$ we call $G_f:=\{
(x,f(x)):x\in X\}$ the graph of $f$ and we say $f:X\to Y$ has closed graph if $G_f$ is a closed subset of $X\times Y$ (with product topology)~\cite{bana}.
The main aim of this text is to characterize all self--maps on 
$n-$dimensional Khalimsky spaces with closed graph.
\subsection*{Background on Alexandroff spaces}
A topological space $(X,\tau)$ or simply $X$ is an Alexandroff space if for each nonempty collection of open subsets  of $X$ like $\Gamma$, $\bigcap\Gamma$ is open too. In particular in Alexandroff space $X$ for each $x\in X$ the intersection of all open neighbourhoods of $X$ is the smallest open neighbourhood of $x$, we denote the smallest open neighbourhood of $x$ by $V_X(x)$ or simply by $V(x)$. 
\\
Moreover, product of two topological Alexandroff spaces is an Alexandroff space too. Note that if $X, Y$ are Alexandroff spaces and $(x,y)\in X\times Y$, then $V_{X\times Y}(x,y)=V_X(x)\times V_Y(y)$ (where 
$X\times Y$ is equipped with product topology). Using induction product of finitely of finitely many Alexandroff spaces is an Alexandroff space.  

\subsection*{Background on Khalimsky spaces}
Let's equip $\mathbb{Z}=\{0,\pm1,\pm2,\cdots\}$ with 
topology $\kappa$ generated by basis $\{\{2n-1,2n,2n+1\}:n\in \mathbb{Z}\}\cup \{\{2n+1\}:n\in \mathbb{Z}\}$. We call the Alexandroff topological space $(\mathbb{Z},\kappa)$ Khalimsky line and denote it by $\mathcal{K}$. For integers $a,b$ with $a\leq b$ let 
\begin{itemize}
\item[$\mathcal A$)] $[a,b]_\mathcal{K}:=\{x\in \mathbb{Z}: a\leq x \leq b\}$,
\item[$\mathcal B$)] $[a,+\infty)_\mathcal{K}:=\{x\in \mathbb{Z}: x\leq a\}$,
\item[$\mathcal C$)] $(-\infty,a]_\mathcal{K}:=\{x\in \mathbb{Z}: x\geq a\}$.
\end{itemize}
All subsets of $\mathcal{K}$ introduced in $(\mathcal{A},\mathcal{B},\mathcal{C})$ and $\varnothing,\mathcal{K}$ are intervals of Khalimsky line  $\mathcal{K}$. 
\\
For natural number $n$, we call $\mathcal{K}^n$ an $n-$dimensional Khalimsky space. 
$\mathcal{K}^2$ is called Khalimsky plane and one may consider it in digital topology too~\cite{dig}.\\
In topological space $X$ suppose $\infty \notin X$ and let $A(X):=X\cup \{\infty\}$. Equip $A(X)$ with topology 
$\{U\subseteq X: U$ is an open subset of $X\}\cup \{U\subseteq A(X):X\setminus U$ is a compact closed subset of 
$X\}$ then we call $A(X)$ one point compactification or Alexandroff compactification of $X$~\cite{see}.
\\
We call one point compactification of Khalimsy line as khalimsky circle and also one point compactification of Khalimsky plane as Khalimsky sphere (see e.g., \cite{semi, one}).
\section{when sel--map $f:\mathcal{K}^n\to \mathcal{K}^n$ has closed graph?} 
\noindent In this section we investigate all cases which a self--map on $\mathcal{K}^n$ (resp.
$A(\mathcal{K}^n)$) has closed graph.
\begin{lemma}\label{salam10}
For each $x,y\in\mathcal{K}^n$ there exists finite sequence $x=x_1,x_2,\ldots,x_p=y\in\mathcal{K}^n$ such that
for each $i\in\{1,\ldots,p-1\}$ we have $V_{\mathcal{K}^n}(x_i)\cap V_{\mathcal{K}^n}(x_{i+1})\neq\varnothing$.
\end{lemma}
\begin{proof}
Let's use induction on $n$. 
\\
If $n=1$ and  $x,y\in\mathcal{K}$ choose $z,w\in\mathbb{Z}$ such that $|2z-x|\leq1$ and $|2w-y|\leq1$, hence
$V_{\mathcal K}(x)\cap V_{\mathcal K}(2z)\neq\varnothing $ and $V_{\mathcal K}(2w)\cap V_{\mathcal K}(y)\neq\varnothing$. 
We may suppose $z\leq w$, then for $x_1=x,x_2=2z, z_3=2z+2,\ldots, x_{w-z+2}=2w,x_{w-z+3}=y$ and each $i\in\{1,\ldots,x-z+2\}$
we have 
$V_{\mathcal K}(x_i)\cap V_{\mathcal K}(x_{i+1})\neq\varnothing$.
\\
Suppose the statement is true  for $n=q\geq1$, i.e. for each $a,b\in\mathcal{K}^q$ there exists finite sequence $a=x_1,x_2,\ldots,x_p=b\in\mathcal{K}^q$ such that
for each $i\in\{1,\ldots,p-1\}$ we have $V_{\mathcal{K}^q}(x_i)\cap V_{\mathcal{K}^q}(x_{i+1})\neq\varnothing$. Choose
$x=(c,d),y=(u,v)\in\mathcal{K}\times\mathcal{K}^q(=\mathcal{K}^{q+1})$, since $c,u\in\mathcal{K}$ and $d,v\in\mathcal{K}^q$
using induction's hypothesis there exist
$c_1=c,\ldots, c_s=u\in\mathcal{K}$ and $d_1=d,\ldots,d_t=v\in\mathcal{K}^q$ such that for each $i\in\{1,\ldots,s-1\}$ and
$j\in\{1,\ldots,t-1\}$  we have $V_{\mathcal K}(c_i)\cap V_{\mathcal K}(c_{i+1})\neq \varnothing$
and $V_{{\mathcal K}^q}(d_j)\cap V_{{\mathcal K}^q}(d_{j+1})\neq \varnothing$. Let
$x_1=x=(c,d)=(c_1,d), x_2=(c_2,d),\ldots,x_s=(c_s,d)=(u,d)=(u,d_1), x_{s+1}=(u,d_2),\ldots, x_{s+t-1}=(u,d_t)=(u,v)$, then
for each $i\in\{1,\ldots,s+t-2\}$ we have $V_{\mathcal{K}^{q+1}}(x_i)\cap V_{\mathcal{K}^{q+1}}(x_{i+1})\neq\varnothing$
(use the fact that $V_{\mathcal{K}^{q+1}}(c_k,d_l)=V_{\mathcal{K}}(c_k)\times V_{\mathcal{K}^{q}}(d_l)$.
\end{proof}
\begin{lemma}\label{salam20}
In topological space $X$ if $f:X\to X$ has closed graph, then for each $x\in X$, and $y\in\overline{\{x\}}$ we have $f(x)=f(y)$, in
particular if $x$ has the smallest open neighbourhood like $V$, then for each $z\in V$, $f(x)=f(z)$.
\end{lemma}
\begin{proof}
Suppose $f:X\to X$ has closed graph, $x\in X$ and $y\in\overline{\{x\}}$. Thus
$(y,f(x))\in\overline{\{x\}}\times\overline{\{f(x)\}}=\overline{\{(x,f(x))\}}\subseteq\overline{G_f}=G_f=\{w,f(w)):w\in X\}$, which leads to
$f(y)=f(x)$. In order to complete the proof note that if $x$ has the smallest open neighbourhood like $V$, then for each $z\in V$
we have $x\in\overline{\{z\}}$, therefore $f(z)=f(x)$.
\end{proof}
\begin{theorem}\label{salam30}
If $n\geq1$ and $X$ is one of the following spaces:
\begin{itemize}
\item an interval in $\mathcal{K}$, 
\item $n-$dimensional Khalimsky space $\mathcal{K}^n$,
\item $A(\mathcal{K}^n)$,
\end{itemize}
then each $f:X\to X$ with closed graph is a constant map.
\end{theorem}
\begin{proof} Consider $f:X\to X$ has closed  graph. \\
First suppose $X$ is an interval in $\mathcal{K}$ with at least two elements or $X$ is $\mathcal{K}^n$, then by Lemma~\ref{salam10} for each 
$x,y\in X$ there exist $x=x_1,\ldots,x_p=y\in X$ such that for each $i\in\{1,\ldots,p-1\}$ we have $V_X(x_i)\cap V_X(x_{i+1})\neq\varnothing$.
By Lemma~\ref{salam20}, $f(x)=f(y)$, which shows that $f$ is a constant map.
\\
Now Suppose $X=A(\mathcal{K}^n)$, using a similar method described in the previous paragraph the restriction of $f$ to $\mathcal{K}^n$ is constant. 
Suppose $f(x)=c$ for all $x\in\mathcal{K}^n$, then $(\infty,c)\in A(\mathcal{K}^n)\times{\{c\}}=\overline{\mathcal{K}^n\times\{c\}}
\subseteq\overline{G_f}=G_f$, hence $f(\infty)=c$ and $f$ is a constant map in this case too.
\end{proof}
\begin{note}\label{salam40}
In nonempty topological spaces $X,Y$ if constant map $\mathop{f:X\to Y}\limits_{x\mapsto c}$ has closed graph, then
$X\times\{c\}=G_f=\overline{G_f}=\overline{X\times\{c\}}=X\times\overline{\{c\}}$, hence $\overline{\{c\}}=\{c\}$ and
$c$ is a closed point of $X$. Thus constant map $\mathop{X\to Y}\limits_{x\mapsto b}$ has closed graph if and only if $b$ is a closed point
of $X$.
\end{note}
\noindent Now we are ready to establish our main theorem.
\begin{theorem}[main]\label{salam40}
Consider $n\geq1$.
\begin{itemize}
\item Suppose $X$ is an interval in $\mathcal{K}$ with at least two elements. $f:X\to X$ has closed graph if and only if there exists
	even integer $2\lambda\in X$ such that $f$ is the constant map with value $2\lambda$.
\item $f:\mathcal{K}^n\to \mathcal{K}^n$ has closed graph if and only if there exist
	even integers $2\lambda_1,\ldots,2\lambda_n$ such that $f$ is the constant map with value $(2\lambda_1,\ldots,2\lambda_n)$.
\item $f:A(\mathcal{K}^n)\to A(\mathcal{K}^n)$ has closed graph if and only if one of the following conditions occurs: 
	\begin{itemize}
	\item 
	there exist
	even integers $2\lambda_1,\ldots,2\lambda_n$ such that $f$ is the constant map with value $(2\lambda_1,\ldots,2\lambda_n)$,
	\item $f$ is the constant map with value $\infty$.
	\end{itemize}
\end{itemize}
\end{theorem}
\begin{proof}
If $X$ is an interval in $\mathcal{K}$ with at least two elements, then $X\cap2\mathbb{Z}$ is the collection of all closed points of $X$.
Also $(2\mathbb{Z})^n$ is the collection of all closed points of $\mathcal{K}^n$. Moreover 
$(2\mathbb{Z})^n\cup\{\infty\}$ is the collection of all closed points of $A(\mathcal{K}^n)$. Now use Theorem~\ref{salam30} and Note~\ref{salam40}
to complete  the proof.
\end{proof}
\section{A diagram}
\noindent In this section via a diagram we compare the collections of self--maps on $\mathcal{K}^n$ 
satisfying on of the following conditions:
\begin{itemize}
\item has closed graph,
\item is constant,
\item is continuous,
\item is quasi--continuous,
\item is just a self--map on $\mathcal{K}^n$.
\end{itemize}
Let's recall that in topological spaces $X,Y$, $f:X\to Y$ is quasi--continuous at $a\in X$ if for each open neighbourhoods $U$ of $a$ and $V$ of $f(a)$ there exists nonempty open subset $W$  of $U$
such that $f(W)\subseteq V$. Also we say $f:X\to Y$ is quasi--continuous if it is quasi--continuous at all pints of$X$~\cite{shir, hola}.
\\
Now we have the following diagram:
\vspace*{5mm}
{\small
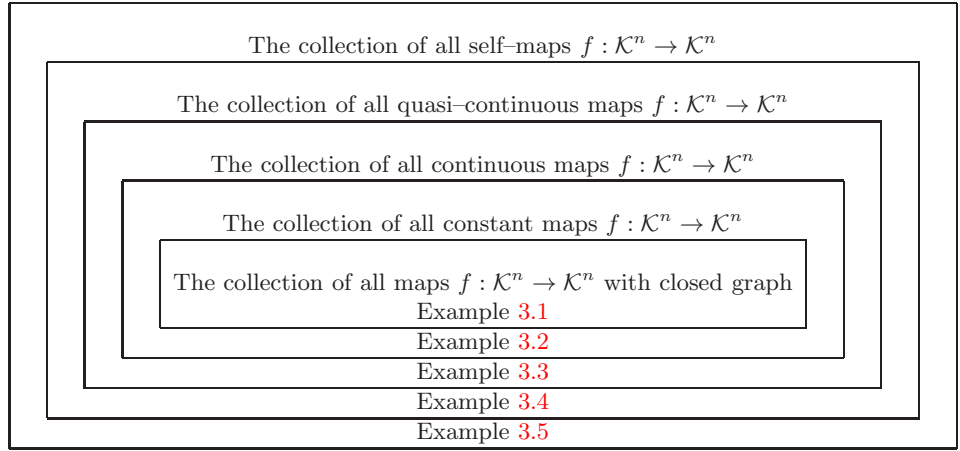
\begin{center}
\begin{tabular}{|c|}
\hline \\
The collection of all self--maps $f:\mathcal{K}^n\to\mathcal{K}^n$ \\
	$\SP$ \begin{tabular}{|c|} \hline  \\
	The collection of all quasi--continuous maps $f:\mathcal{K}^n\to\mathcal{K}^n$ \\
		$\SP$ \begin{tabular}{|c|} \hline  \\
		The collection of all continuous maps $f:\mathcal{K}^n\to\mathcal{K}^n$ \\
			$\SP$ \begin{tabular}{|c|} \hline  \\
			The collection of all constant maps $f:\mathcal{K}^n\to\mathcal{K}^n$ \\
				$\SP$ \begin{tabular}{|c|} \hline  \\
				The collection of all maps $f:\mathcal{K}^n\to\mathcal{K}^n$ with closed graph
				\\ Example~\ref{A}
				\\ \hline \end{tabular} $\SP$ \\
			Example~\ref{B}
			\\ \hline \end{tabular} $\SP$ \\
		Example~\ref{C}
		\\ \hline \end{tabular} $\SP$ \\
	Example~\ref{D}
	\\ \hline \end{tabular} $\SP$ \\
Example~\ref{E}
\\ \hline 
\end{tabular}
\vspace{5mm}
\end{center}}
\noindent In order to complete the above diagram consider the following examples which have been referred in the diagram.
\begin{example}\label{A}
Consider $f:\mathcal{K}^n\to\mathcal{K}^n$  with $f(x)=(0,\cdots,0)$ for all $x\in\mathcal{K}^n$, then by Theorem~\ref{salam40},
$f:\mathcal{K}^n\to\mathcal{K}^n$ has closed graph. Note that by Theorem~\ref{salam40} all self--maps on $\mathcal{K}^n$ with closed graph
are constant maps.
\end{example}
\begin{example}\label{B}
Consider constant map $f:\mathcal{K}^n\to\mathcal{K}^n$  with $f(x)=(1,\cdots,1)$ for all $x\in\mathcal{K}^n$
then by Theorem~\ref{salam40},
$f:\mathcal{K}^n\to\mathcal{K}^n$ does not have closed graph. Note that all constant maps are continuous.
\end{example}
\begin{example}\label{C}
$f:\mathcal{K}^n\to\mathcal{K}^n$  with $f(x)=x$ for all $x\in\mathcal{K}^n$ is a continuous non--constant map. Note that all
continuous maps are quasi--continuous.
\end{example}
\begin{example}\label{D}
Consider $h:\mathcal{K}\to\mathcal{K}$ with $h(2m)=h(2m+1)=2m+1$ for each $m\in\mathbb{}$, then 
$f:\mathcal{K}^n\to\mathcal{K}^n$  with $f(x_1,\cdots,x_n)=(h(x_1),\cdots,h(x_n))$ for all $(x_1,\cdots,x_n)\in\mathcal{K}^n$
is a quasi--continuous map which is not continuous.
\end{example}
\begin{example}\label{E}
Consider $h:\mathcal{K}\to\mathcal{K}$ with $h(2m)=-h(2m+1)=1$ for each $m\in\mathbb{}$, then 
$f:\mathcal{K}^n\to\mathcal{K}^n$  with $f(x_1,\cdots,x_n)=(h(x_1),\cdots,h(x_n))$ for all $(x_1,\cdots,x_n)\in\mathcal{K}^n$
is a is not quasi--continuous.
\end{example}
\section*{Acknowledgment} 
\noindent The authors wish to dedicate this paper to our lady,
Fatimah-Zahra (as).

\noindent \noindent {\small {\bf Mehrnaz Pourattar} Department of Mathematics, Science and Research Branch, Islamic Azad University, Tehran, Iran
 (mpourattar@yahoo.com)}
\\
{\small {\bf Fatemah Ayatollah Zadeh Shirazi}, Faculty
of Mathematics, Statistics and Computer Science, College of
Science, University of Tehran, Tehran, Iran
 (f.a.z.shirazi@ut.ac.ir)}
\\
{\small {\bf Mohammad Reza Mardanbeigi}, Department of Mathematics, Science and Research Branch, Islamic Azad University, Tehran, Iran
 (mrmardanbeigi@srbiau.ac.ir)}

\end{document}